\documentclass[a4paper, 11pt, oneside, notitlepage]{amsart}

\usepackage{amsmath,amscd}
\usepackage{amssymb}
\usepackage{amsthm}
\usepackage{comment}
\usepackage{graphicx, xcolor}

\usepackage{mathrsfs}
\usepackage[ocgcolorlinks, linkcolor=blue]{hyperref}

\usepackage[ocgcolorlinks,linkcolor=blue]{hyperref}
\usepackage{xcolor}

\theoremstyle{plain}
\newtheorem{thm}{Theorem}[section]
\newtheorem{prop}{Proposition}[section]
\newtheorem{lem}[prop]{Lemma}
\newtheorem{cor}[prop]{Corollary}

\newtheorem*{thm*}{Theorem}

\theoremstyle{definition}
\newtheorem{rmk}[prop]{Remark}

\numberwithin{equation}{section}
\newcommand {\R} {\mathbb{R}}

\newcommand {\p} {\partial}

\newcommand{\eps}{\varepsilon}

\newcommand {\supp} {\text{supp}}

\newcommand{\ol}{\overline}

\newcommand{\tbl}{\textcolor{blue}}

\newcommand{\mR}{\mathbb{R}}                    % Formatting for R
\newcommand{\mC}{\mathbb{C}}                    % Formatting for C
\newcommand{\abs}[1]{\lvert #1 \rvert}          % Formatting for the absolute value
\newcommand{\norm}[1]{\left\Vert #1 \right\Vert} 
\newcommand{\re}{\mathrm{Re}}
\newcommand{\im}{\mathrm{Im}}

\definecolor{armygreen}{rgb}{0.29, 0.33, 0.13}
\definecolor{ao(english)}{rgb}{0.0, 0.5, 0.0}

\title[Inverse problems for semilinear elliptic PDE]{Inverse problems for semilinear elliptic PDE with measurements at a single point}

\author[Salo]{Mikko Salo}
\address{Department of Mathematics and Statistics, University of Jyvaskyla, Jyvaskyla, Finland}
\curraddr{}
\email{mikko.j.salo@jyu.fi}

\author[Tzou]{Leo Tzou}
\address{School of Mathematics and Statistics, University of Sydney, Australia}
\curraddr{}
\email{leo.tzou@gmail.com}

\begin{document}

\maketitle
\begin{abstract}
We consider the inverse problem of determining a potential in a semilinear elliptic equation from the knowledge of the Dirichlet-to-Neumann map. For bounded Euclidean domains we prove that the potential is uniquely determined by the Dirichlet-to-Neumann map measured at a single boundary point, or integrated against a fixed measure. This result is valid even when the Dirichlet data is only given on a small subset of the boundary. We also give related uniqueness results on Riemannian manifolds.

%		\medskip
	
%	\noindent{\bf Keywords.} Inverse boundary value problem, Calder\'on problem, partial data, semilinear elliptic equations, higher order linearization, transversally anisotropic manifold.
	
	%\noindent{\bf Mathematics Subject Classification (2010)}: 
\end{abstract}

\section{Introduction}\label{Sec 1}

In this article we study inverse problems for semilinear elliptic equations, with measurements given by the nonlinear Dirichlet-to-Neumann map (DN map) measured at a single point or integrated against a fixed measure. The method is based on higher order linearizations of the DN map. This method was introduced in inverse problems for hyperbolic PDE in \cite{KLU2018} where a source-to-solution map was used. It was observed in \cite{LLPT} that in the hyperbolic case it may be sufficient to measure a DN map integrated against a suitable fixed function. The work \cite{TzouSingle} proved a result showing that measurements of the source-to-solution map at a single point suffice. %The survey \cite{Lassas_survey} contains further references to results in nonlinear hyperbolic PDEs.

The higher order linearization method in inverse problems for nonlinear elliptic PDE was introduced independently in \cite{FO20} and \cite{LLLS2019nonlinear}. We note that the first linearization has been used extensively since the work \cite{isakov1993uniqueness_parabolic}, and the second linearization had also been used in \cite{sun1996quasilinear,sun1997inverse,KN002,CNV2019reconstruction,AYT2017direct}. The works \cite{LLLS2019partial,KU2019partial,KU2019remark} studied related inverse problems for semilinear elliptic equations with partial data, with \cite{LLST} addressing fractional power nonlinearities. %, and \cite{LL2020inverse,lin2020monotonicity,LO2020fractional_lower} considered fractional semilinear elliptic equations.
In \cite{LZ2020partial,krupchyk2020inverse,carstea2020inverse,kian2020partial, CFKKU} the authors study nonlinear conductivity or magnetic Schr\"odinger type equations. All these results use the nonlinear DN map with data given on open subsets of the boundary.

In this note we observe that in some of the elliptic results above it is enough to measure the DN map at a single point, or integrated against a fixed measure. Let $\Omega \subset \R^n$, $n \geq 2$, be a bounded domain with $C^\infty$ boundary, and let $m \geq 2$ be an integer. Consider the semilinear elliptic equation 
\begin{align}\label{Main equation}
	\left\lbrace \begin{array}{rl}
	\Delta u + q(x) u^m =0 & \text{ in }\Omega, \\[3pt]
	u=f & \text{ on }\p \Omega,
	\end{array} \right.
\end{align}
where $q\in C^\alpha(\overline{\Omega})$ is a potential, and $C^\alpha$ with $0 < \alpha < 1$ denotes the space of $\alpha$-H\"older continuous functions. Let $f \in U_{\delta}$, where 
\[
U_{\delta} := \{ f \in C^{2,\alpha}(\p \Omega) \,:\, \norm{f}_{C^{2,\alpha}(\p \Omega)} < \delta \}.
\]
If $\delta > 0$ is small enough there is a unique small solution $u = u_f \in C^{2,\alpha}(\ol{\Omega})$ of \eqref{Main equation}, see e.g.\ \cite[Proposition 2.1]{LLST}. One can then define the corresponding nonlinear DN map $\Lambda_q$ by 
\begin{align*}
	\Lambda_q: U_{\delta} \to C^{1,\alpha}(\p \Omega),\qquad  f \mapsto \left. \p_\nu u _f \right|_{\p \Omega},
\end{align*}
where $\p_\nu$ denotes the normal derivative on $\p \Omega$. In \cite{FO20, LLLS2019nonlinear} it was proved that the full DN map $\Lambda_q$ uniquely determines $q$. This was extended in \cite{KU2019remark,LLLS2019partial} to the case where one knows $\Lambda_q(f)|_{\Gamma_1}$ for $f$ supported in $\Gamma_2$ where $\Gamma_1, \Gamma_2 \subset \p \Omega$ are open sets.

We show that it is enough to measure $\int_{\p \Omega} \Lambda_q(f) \,d\mu$ for a fixed measure $\mu$ on $\p \Omega$. When $\mu = \delta_{x_0}$ this corresponds to measurements at a fixed point.

\begin{thm} \label{Main Thm 1}
Let $\Omega \subset \R^n$, $n \geq 2$, be a connected bounded open set with $C^\infty$ boundary, let $m \geq 2$ be an integer, and let $\Gamma \subset \p \Omega$ be a nonempty open set. Suppose that $\mu \not\equiv 0$ is a fixed measure on $\p \Omega$. If $q_1, q_2 \in C^\alpha (\overline{\Omega})$ for some $0<\alpha<1$ satisfy  
\begin{equation} \label{dnmap_measure_assumption}
\int_{\p \Omega} \Lambda_{q_1}(f) \,d\mu = \int_{\p \Omega }\Lambda_{q_2}(f) \,d\mu
\end{equation}
for all $f\in U_{\delta}$ with $\mathrm{supp}(f) \subset \Gamma$ where $\delta >0$ is sufficiently small, then 
\begin{align*}
 q_1 =q_2 \text{ in }\Omega.
\end{align*}
In particular, choosing $\mu = \delta_{x_0}$ for some fixed $x_0 \in \p \Omega$, we see that the condition 
\[
 \Lambda_{q_1}(f)(x_0) = \Lambda_{q_2}(f)(x_0) \qquad \text{for all $f\in U_{\delta}$ with $\mathrm{supp}(f) \subset \Gamma$}
\]
implies that $q_1 = q_2$.
\end{thm}

We can give a similar result for semilinear elliptic PDE on manifolds. Let $(M,g)$ be a compact Riemannian manifold with smooth boundary, let $q \in C^{\infty}(M)$, and let $m \geq 2$. We consider the Dirichlet problem 
\begin{align}\label{eq_manifold}
	\left\lbrace \begin{array}{rl}
	\Delta_g u + q(x) u^m =0 & \text{ in }M, \\[3pt]
	u=f & \text{ on }\p M.
	\end{array} \right.
\end{align}
Again, if $U_{\delta} := \{ f \in C^{2,\alpha}(\p M) \,:\, \norm{f}_{C^{2,\alpha}(\p M)} < \delta \}$, then for any $f \in U_{\delta}$ with $\delta$ small enough the Dirichlet problem has a unique small solution $u \in C^{2,\alpha}(M)$ (see e.g.\ \cite[Proposition 2.1]{LLLS2019nonlinear}). We may define the DN map 
\begin{align*}
	\Lambda_q: U_{\delta} \to C^{1,\alpha}(\p M),\qquad  f \mapsto \left. \p_\nu u _f \right|_{\p M},
\end{align*}
where $\p_{\nu}$ denotes the normal derivative with respect to the metric $g$ on $\p M$. We have the following result where $f$ can be supported on all of $\p M$, but we only measure the DN map at a single point or integrated against a fixed measure.

\begin{thm} \label{thm2}
Let $(M,g)$ be a compact Riemannian $n$-manifold with smooth boundary, let $m \geq 2$ be an integer, and let $\mu \not\equiv 0$ be a fixed measure on $\p M$. Assume that one of the following conditions is satisfied:
\begin{enumerate}
\item 
$(M,g)$ is transversally anisotropic as in \cite[Definition 1.1]{LLLS2019nonlinear}, and $m \geq 4$; or 
\item 
$(M,g)$ is a complex manifold satisfying the conditions in \cite[Theorem 1.4]{GuillarmouSaloTzou}.
\end{enumerate}
If $q_1, q_2 \in C^{\infty} (M)$ are such that  $q_1 = q_2$ to infinite order on $\p M$ and 
\begin{equation} \label{dnmap_measure_assumption2}
\int_{\p \Omega} \Lambda_{q_1}(f) \,d\mu = \int_{\p \Omega }\Lambda_{q_2}(f) \,d\mu
\end{equation}
for all $f\in U_{\delta}$ where $\delta >0$ is sufficiently small, then $q_1 = q_2$ in $M$.
\end{thm}

The proofs of Theorems \ref{Main Thm 1}--\ref{thm2} are based on the higher order linearization method in \cite{FO20, LLLS2019nonlinear}. From \cite[Proposition 2.2]{LLLS2019nonlinear} one obtains the identity 
\begin{multline}
\int_{\p M} ((D^m \Lambda_{q_1})_0 - (D^m \Lambda_{q_2})_0)(f_1, \ldots, f_m) f_{m+1} \,dS \\
 = -(m!) \int_M (q_1-q_2) v_1 \cdots v_{m+1} \,dV \label{dmlambda_identity}
\end{multline}
where $(D^m \Lambda_q)_0$ denotes the $m$th Fr\'echet derivative on $\Lambda_q$ at $0$ considered as an $m$-linear form, $f_j$ are Dirichlet data, and $v_j$ are solutions of the linearized equation $\Delta_g v_j = 0$ in $M$ with $v_j|_{\p M} = f_j$. The single point measurement case formally corresponds to choosing $f_{m+1} = \delta_{x_0}$ with $x_0 \in \p M$. The corresponding solution $v_{m+1}$ is in $L^1(\Omega)$ but it is not bounded, and this will require some additional arguments.

If one has equality of the DN maps for $q_1$ and $q_2$ as in Theorems \ref{Main Thm 1}--\ref{thm2}, the identity \eqref{dmlambda_identity} implies that 
\[
\int_M f v_1 v_2 \,dV = 0
\]
where $f := (q_1-q_2) v_3 \cdots v_m v_{m+1}$ and $v_j$ are as above. We choose $v_3, \ldots, v_m$ to be smooth nonvanishing solutions, and $v_{m+1}$ will be the (nonvanishing) $L^1(\Omega)$ solution whose Dirichlet data is a measure. It is then enough to show that $f=0$, which will imply $q_1=q_2$. For the partial data result in Theorem \ref{Main Thm 1}, we need the following extension given in \cite[Section 4]{CarsteaGhoshUhlmann} of the fundamental result of \cite{ferreira2009linearized} on the linearized local Calder\'on problem that was originally proved for $f \in L^{\infty}(\Omega)$. % (we expect that this holds for $f \in L^1(\Omega)$ as well, a possible argument is indicated in Section \ref{sec_linearized_proof}).

\begin{thm} \label{thm_linearized_lone}
Let $\Omega \subset \R^n$, $n \geq 2$, be a connected bounded open set with $C^\infty$ boundary, and let $\Gamma \subset \p \Omega$ be a nonempty open set. Suppose that $f \in L^1(\Omega)$ is such that 
\[
\int_{\Omega} f v_1 v_2 \,dx = 0
\]
for all $v_j \in C^{\infty}(\ol{\Omega})$ solving $\Delta v_j = 0$ in $\Omega$ with $\mathrm{supp}(v_j|_{\p \Omega}) \subset \Gamma$. Then $f=0$ in $\Omega$.
\end{thm}

For Theorem \ref{thm2} we will invoke the results in \cite{LLLS2019nonlinear, GuillarmouSaloTzou} instead.

\subsection*{Acknowledgments.}

M.S.\ was partly supported by the Academy of Finland (Centre of Excellence in Inverse Modelling and Imaging, grant 284715) and by the European Research Council under Horizon 2020 (ERC CoG 770924). L.T.\ was partly supported by Australian Research Council Discovery Projects DP190103451 and DP190103302.

\section{Proof of Theorem \ref{Main Thm 1}}

For the proof of Theorem \ref{Main Thm 1}, we give a lemma related to solving the Dirichlet problem when the boundary value is a finite Borel measure $\mu$ on $\p \Omega$. We use the norm given by the total variation, 
\[
\norm{\mu}_{\mathcal{M}(\p \Omega)} = \abs{\mu}(\p \Omega) = \sup_{\norm{\varphi}_{C(\p \Omega) = 1}} \,\left\lvert \int_{\p \Omega} \varphi \,d\mu \right\rvert.
\]
We need the fact that the solution is in $L^r(\Omega)$ for $1 \leq r < \frac{n}{n-1}$.

\begin{lem} \label{lemma_dirichlet_measure}
Let $\Omega \subset \R^n$, $n \geq 2$, be a bounded open set with $C^\infty$ boundary, and let $\mu$ be a finite complex Borel measure on $\p \Omega$. Consider the function 
\[
\Psi(x) = \int_{\p \Omega} P(x,y) \,d\mu(y), \qquad x \in \Omega,
\]
where $P(x,y)$ is the Poisson kernel for $\Delta$ in $\Omega$. Then $\Psi \in L^r(\Omega)$ where $1 \leq r < \frac{n}{n-1}$, and it solves the Dirichlet problem 
\begin{align}\label{psi_equation}
	\left\{ \begin{array}{rll}
	\Delta \Psi &\!\!\!=0 & \text{ in }\Omega, \\[3pt]
	\Psi&\!\!\!= \mu& \text{ on }\p \Omega
	\end{array} \right.
\end{align}
where the boundary value is understood as follows: for any $w \in C^2(\ol{\Omega})$ with $w|_{\p \Omega} = 0$ one has 
\begin{equation} \label{dirichlet_bv_measure}
\int_{\p \Omega} \p_{\nu} w \,d\mu = \int_{\Omega} (\Delta w) \Psi \,dx.
\end{equation}
\end{lem}
\begin{proof}
By applying a partition of unity, boundary flattening transformations and convolution approximation, we can produce a sequence $\psi_j \in C^{\infty}(\p \Omega)$ such that $\norm{\psi_j \,dS - \mu}_{\mathcal{M}(\p \Omega)} \to 0$. Let $\Psi_j \in C^{\infty}(\ol{\Omega})$ solve $\Delta \Psi_j = 0$ in $\Omega$ with $\Psi_j|_{\p \Omega} = \psi_j$. If $w$ is as in the statement of the lemma, integration by parts gives 
\[
\int_{\p \Omega} (\p_{\nu} w) \psi_j \,dS = \int_{\Omega} (\Delta w) \Psi_j \,dx.
\]

It is thus sufficient to show that $\Psi \in L^r(\Omega)$ and $\Psi_j \to \Psi$ in $L^r(\Omega)$ for $1 \leq r < \frac{n}{n-1}$. We apply the Poisson kernel estimate (see e.g.\ \cite{Krantz2005}) 
\[
P(x,y) \leq \frac{C \,\mathrm{dist}(x, \p \Omega)}{\abs{x-y}^n} \leq \frac{C}{\abs{x-y}^{n-1}}
\]
for some $C > 0$. If $\Omega_{\delta} = \{ x \in \Omega \,:\, \mathrm{dist}(x, \p \Omega) > \delta \}$, the Minkowski inequality in integral form gives 
\begin{align*}
\norm{\Psi(x)}_{L^r(\Omega_{\delta})} &\leq \int_{\p \Omega} \norm{P(\,\cdot\,,y)}_{L^r(\Omega_{\delta})} \,d\abs{\mu}(y) \\
 &\leq \left[ \sup_{y \in \p \Omega} \left( \int_{\Omega_{\delta}} \frac{C}{\abs{x-y}^{(n-1)r}} \,dx \right)^{1/r} \right]  \norm{\mu}_{\mathcal{M}(\p \Omega)}.
\end{align*}
The quantity in brackets is finite uniformly over $\delta > 0$ when $r < \frac{n}{n-1}$. Thus we may let $\delta \to 0$ to obtain that $\Psi \in L^r(\Omega)$. Applying the same argument to 
\[
\Psi_j(x) - \Psi(x) = \int_{\p \Omega} P(x,y) (\psi_j(y) \,dS(y) - d\mu(y))
\]
shows that $\Psi_j \to \Psi$ in $L^r(\Omega)$.
\end{proof}

\begin{proof}[Proof of Theorem \ref{Main Thm 1}]
Let first $q \in C^{\alpha}(\ol{\Omega})$ be fixed. Consider Dirichlet data of the form $f_{\eps} = \eps_1 h_1 + \ldots + \eps_m h_m$ where $h_j \in C^{\infty}(\p \Omega)$ satisfy $\mathrm{supp}(h_j) \subset \Gamma$, and $\eps = (\eps_1, \ldots, \eps_m)$ where $\eps_j$ are sufficiently small. Let $u_{\eps}$ be the solution of \eqref{Main equation} with Dirichlet data $f_{\eps}$. By \cite[Proposition 2.1]{LLST} the map $\eps \mapsto u_{\eps}$ is smooth. By uniqueness of small solutions one has $u_0 = 0$, and by differentiating \eqref{Main equation} with respect to $\eps_j$ one has $\p_{\eps_j} u_{\eps}|_{\eps=0} = v_j$ where $v_j$ is the solution of  
\begin{align}\label{vj_equation}
	\left\{ \begin{array}{rll}
	\Delta v_j &\!\!\!=0 & \text{ in }\Omega, \\[3pt]
	v_j&\!\!\!=h_j & \text{ on }\p \Omega.
	\end{array} \right.
\end{align}
Moreover, applying $\p_{\eps_1} \ldots \p_{\eps_m}$ to \eqref{Main equation} and evaluating at $\eps = 0$ implies that $w := \p_{\eps_1} \ldots \p_{\eps_m} u_{\eps}|_{\eps=0}$ solves the equation 
\begin{align}\label{w_equation}
	\left\{ \begin{array}{rll}
	\Delta w &\!\!\!= -(m!) q v_1 \cdots v_m & \text{ in }\Omega, \\[3pt]
	w&\!\!\!=0 & \text{ on }\p \Omega.
	\end{array} \right.
\end{align}
By elliptic regularity, $v_j \in C^{\infty}(\ol{\Omega})$ and $w \in C^{2,\alpha}(\ol{\Omega})$. The DN map satisfies 
\begin{equation} \label{dnmap_w}
\p_{\eps_1} \ldots \p_{\eps_m} (\Lambda_q(f_{\eps}))|_{\eps = 0} = \p_{\eps_1} \ldots \p_{\eps_m} (\p_{\nu} u_{\eps})|_{\eps=0} = \p_{\nu} w|_{\p \Omega}.
\end{equation}

Now assume that $q_1, q_2 \in C^{\alpha}(\ol{\Omega})$ are such that \eqref{dnmap_measure_assumption} holds. Let $w_j$ be the solution of \eqref{w_equation} for $q = q_j$. By \eqref{dnmap_measure_assumption} and \eqref{dnmap_w}, one has 
\[
\int_{\p \Omega} \p_{\nu}(w_1-w_2) \,d\mu = 0.
\]
Let $\Psi \in L^r(\Omega)$ with $r < \frac{n}{n-1}$ be the solution of $\Delta \Psi = 0$ in $\Omega$ with $\Psi|_{\p \Omega} = \mu$ in the sense of Lemma \ref{lemma_dirichlet_measure}. It follows from \eqref{dirichlet_bv_measure} that 
\[
0 = \int_{\Omega} \Delta(w_1 - w_2) \Psi \,dx = -(m!) \int_{\Omega} (q_1-q_2) v_1 \ldots v_m \Psi \,dx.
\]
Now choose $h_3, \ldots, h_m \in C^{\infty}(\p \Omega)$ so that $\supp(h_j) \subset \Gamma$, $h_j \geq 0$, and $h_j > 0$ somewhere. By the strong maximum principle $v_j > 0$ in $\Omega$ for $3 \leq j \leq m$. We obtain that 
\begin{equation} \label{orthogonality}
\int_{\Omega} [(q_1-q_2) v_3 \cdots v_m \Psi] v_1 v_2 \,dx = 0
\end{equation}
for any $h_1, h_2 \in C^{\infty}(\p \Omega)$ with $\mathrm{supp}(h_j) \subset \Gamma$. Note that the function in brackets is in $L^r(\Omega)$ for $r < \frac{n}{n-1}$. Now we invoke Theorem \ref{thm_linearized_lone}, which implies that $(q_1-q_2) v_3 \cdots v_m \Psi = 0$ in $\Omega$. Since $v_3, \ldots, v_m$ are positive we must have $(q_1-q_2)\Psi = 0$ in $\Omega$. Finally, since $\mu \not\equiv 0$, the solution $\Psi$ cannot vanish in any open subset of $\Omega$ by unique continuation (otherwise \eqref{dirichlet_bv_measure} would imply that $\mu \equiv 0$). Thus $\Psi$ is nonzero in a dense set of points in $\Omega$. Since $q_j$ are continuous, this shows that $q_1 = q_2$.
\end{proof}

\section{Proof of Theorem \ref{thm2}}

We now describe how to prove Theorem \ref{thm2}. The proof is very similar to that of Theorem \ref{Main Thm 1} and we indicate the required modifications. First we note that Lemma \ref{lemma_dirichlet_measure} extends to the case where $\Omega$ is replaced by a compact Riemannian manifold $(M,g)$ with smooth boundary and $\Delta$ is replaced by $\Delta_g$. This relies on estimates for the Poisson kernel $P(x,y)$ on compact manifolds with boundary:
\begin{equation} \label{pk_est_riem}
\abs{\nabla_x^k P(x,y)} \leq \frac{C_k}{d_g(x,y)^{n-1+k}}, \qquad x \in M, \ y \in \p M.
\end{equation}
In fact the case $k=0$ follows e.g.\ from \cite[Lemma 2.2]{HangWangYan}. The general case follows by writing $\eps = d_g(x,y)$ and by inserting $u(\,\cdot\,) = P(\,\cdot\,,y)$  into the elliptic estimate 
\[
\lVert \nabla^k u \rVert_{L^{\infty}(B_{\eps/4}(x) \cap M)} \leq C_k \eps^{-k} \norm{u}_{L^{\infty}(B_{\eps/2}(x) \cap M)}.
\]
The last estimate is valid by standard elliptic regularity after rescaling into a ball of radius one.

Assuming the conditions in Theorem \ref{thm2}, the same argument that leads to \eqref{orthogonality} yields the identity 
\begin{equation} \label{orthogonality_riem}
\int_M (q_1-q_2) v_1 \cdots v_m \Psi \,dV_g = 0
\end{equation}
where $v_j \in C^{\infty}(M)$ are arbitrary solutions of the equation $\Delta_g v_j = 0$ in $M$, and $\Psi \in L^r(M)$ for $1 \leq r < \frac{n}{n-1}$ is the solution of 
\begin{align*}
	\left\{ \begin{array}{rll}
	\Delta_g \Psi &\!\!\!=0 & \text{ in }M, \\[3pt]
	\Psi&\!\!\!= \mu& \text{ on }\p M.
	\end{array} \right.
\end{align*}

Note that by elliptic regularity, $\Psi$ is smooth in $M^{\mathrm{int}}$ and it is also smooth up to the boundary near points $z \in \p M$ so that $\mu = 0$ near $z$. To study the situation near $\supp(\mu)$, we observe using \eqref{pk_est_riem} that for any $x \in M^{\mathrm{int}}$ one has 
\[
\abs{\Psi(x)} \leq \left\lvert \int_{\p M} P(x,y) \,d\mu(y) \right\rvert \leq C \int_{\p M} \frac{1}{d_g(x,y)^{n-1}} \,d\abs{\mu}(y).
\]
Write $f := (q_1-q_2)\Psi$. Using the assumption that $q_1=q_2$ to infinite order on $\p M$, for any $N \geq 0$ there is $C_N > 0$ such that 
\begin{align*}
\abs{f(x)} &\leq C_N d_g(x, \p M)^{N} \int_{\p M} \frac{1}{d_g(x,y)^{n-1}} \,d\abs{\mu}(y) \\
 &\leq C_N d_g(x, \p M)^{N-(n-1)} |\mu|(\p M).
\end{align*}
Choosing $N \geq n$ gives that $f$ is bounded in $M$ and vanishes on $\p M$. Applying similar estimates to derivatives of $f$ in $M^{\mathrm{int}}$ proves that $f$ is actually $C^{\infty}$ up to the boundary in $M$ and it vanishes to infinite order on $\p M$.

We rewrite \eqref{orthogonality_riem} in the form 
\[
\int_M f v_1 \dots v_m \,dV_g = 0
\]
where $f = (q_1-q_2)\Psi$ and $v_j \in C^{\infty}(M)$ are any solutions of $\Delta_g v_j = 0$ in $M$. It now follows from \cite[Proposition 5.1]{LLLS2019nonlinear}, if $(M,g)$ is transversally anisotropic and $m \geq 4$, or from \cite[Theorem 1.4]{GuillarmouSaloTzou}, if $(M,g)$ is a complex manifold satisfying the assumptions of that theorem, that $f = 0$. Since $\mu \not \equiv 0$ and $M$ is connected, $\Psi$ cannot vanish in any open set in $M^{\mathrm{int}}$ by the unique continuation principle. Thus we must also have $q_1-q_2=0$ in $M$, which concludes the proof of Theorem \ref{thm2}.

\begin{rmk}
Under assumption (1) in Theorem \ref{thm2}, the condition that $q_1=q_2$ to infinite order on $\p M$ can be weakened. In fact it would be enough to suppose that $q_1=q_2$ to suitable finite order near $\supp(\mu)$ on $\p M$, since in that case the argument above shows that $(q_1-q_2) \Psi$ is in $C^1(M)$ and hence \cite[Proposition 5.1]{LLLS2019nonlinear} applies. In a similar vein, under assumption (1) and in the special case $\mu = \delta_{x_0}$, it would be enough to assume that $\nabla^k q_1(x_0) = \nabla^k q_2(x_0)$ for finitely many $k$.
\end{rmk}

\appendix

\section{Proof of Theorem \ref{thm_linearized_lone}} \label{sec_linearized_proof}

%The proof of the partial data result in Theorem \ref{Main Thm 1} relies on the fundamental work \cite{ferreira2009linearized} on the linearized local Calder\'on problem. However, since solutions whose Dirichlet boundary value is a measure are in general only in $L^r(\Omega)$ for $1 \leq r < \frac{n}{n-1}$, we need a version of \cite[Theorem 1.1]{ferreira2009linearized} that applies to functions in $L^r(\Omega)$ instead of $L^{\infty}(\Omega)$.

As mentioned, Theorem \ref{thm_linearized_lone} is proved in \cite[Section 4]{CarsteaGhoshUhlmann}, and the proof relies on a Runge approximation result given in \cite[Lemma 2.6]{kian2020partial}. In this appendix we give a slightly shorter proof for $f \in L^r(\Omega)$ for $r > 1$, which is already sufficient for all the results in this article. Later we also give an alternative argument that works for $f \in L^1(\Omega)$.

We begin with a version of the Runge approximation result given in \cite[Lemma 2.2]{ferreira2009linearized} where the approximation is in the $L^p$ norm where $p$ is large. A stronger result for the $W^{1,p}$ norm is in \cite[Lemma 2.6]{kian2020partial}.

\begin{lem} \label{lemma_runge}
Let $\Omega_1\subset \Omega_2$ be bounded open sets with smooth boundary, and let $1 < p < \infty$. Let $G_{\Omega_2}(x,y)$ be the Dirichlet Green's kernel associated with $\Omega_2$. Then the set 
\[
R = \left\{ \int_{\Omega_2} G_{\Omega_2}(\cdot, y) a(y) \,dy \,:\, a\in C^\infty_c(\Omega_2), \ \mathrm{supp}(a)\subset \Omega_2 \setminus \ol{\Omega}_1 \right\}
\]
is a dense subspace, with respect to the $L^p(\Omega_1)$ topology, in the space $S$ of harmonic functions $u \in C^{\infty}(\ol{\Omega}_1)$ with $u|_{\partial\Omega_1 \cap \partial \Omega_2} = 0$.
\end{lem}
\begin{proof}
Suppose that $\ell$ is a bounded linear functional on $L^p(\Omega_1)$ with $\ell|_R = 0$. We need to show that $\ell|_S = 0$. By duality there is $v \in L^{p'}(\Omega_1)$, where $p'$ is the dual H\"older exponent of $p$, so that  
\[
\ell(u) = \langle v, u \rangle_{\Omega_1}.
\]
Denote by $G$ the solution operator for $\Delta$ in $\Omega_2$ with vanishing Dirichlet data, and write $E_0 v$ for the zero extension of $v$ to $\Omega_2$. The assumption $\ell|_R = 0$ ensures that for any $a\in C^\infty_c(\Omega_2 \setminus \ol{\Omega}_1)$ we have 
\begin{align} \label{vanishing_condition}
0 = \langle v, G a |_{\Omega_1} \rangle_{\Omega_1} = \langle E_0 v, Ga \rangle_{\Omega_2}.
\end{align}

Now, let $w = G(E_0 v) \in W^{2,p'}(\Omega_2)$ solve $\Delta w = E_0 v$ in $\Omega_2$ with $w|_{\p \Omega_2} = 0$. For any $f \in W^{-1,p'}(\Omega_2)$ and $h \in W^{-1,p}(\Omega_2)$, one can check the duality 
statement 
\[
\langle Gf, h \rangle_{\Omega_2} = \langle f, Gh \rangle_{\Omega_2}.
\]
Using this duality statement in \eqref{vanishing_condition}, we obtain that 
\begin{equation} \label{w_vanishing}
w|_{\Omega_2 \setminus \ol{\Omega}_1} = 0.
\end{equation}

Let now $u \in S$, and let $Eu$ be any function in $C^{\infty}(\ol{\Omega}_2)$ with $Eu|_{\Omega_1} = u$. We wish to show that $\ell(u) = 0$. We may compute 
\begin{align}
\ell(u) &= \langle v, u \rangle_{\Omega_1} = \langle E_0 v, Eu \rangle_{\Omega_2} = \langle \Delta w, Eu \rangle_{\Omega_2} \label{runge_ibp}
 \\
 &= \langle \p_{\nu} w, Eu \rangle_{\p \Omega_2} - \langle \nabla w, \nabla(Eu) \rangle_{\Omega_2}. \notag
\end{align}
Here we used that $\nabla w \in W^{1,q}(\Omega_2)$ for $q = \frac{np'}{n-p'}$ when $p' < n$ (and for any $q < \infty$ if $p' \geq n$), showing that $\p_{\nu} w$ is in the Besov space $B^{1-1/q}_{q,q}(\p \Omega_2)$ by the trace theorem \cite[Section 4.7]{Triebel1978}. We integrate by parts once more and use the condition $w|_{\p \Omega_2} = 0$ to obtain that 
\[
\ell(u) = \langle \p_{\nu} w, Eu \rangle_{\p \Omega_2} + \langle w, \Delta(Eu) \rangle_{\Omega_2}.
\]
Since $u \in S$, we have $Eu|_{\p \Omega_1 \cap \p \Omega_2} = 0$ and $\Delta(Eu)|_{\Omega_1} = 0$. On the other hand, \eqref{w_vanishing} implies that $\p_{\nu} w|_{\p \Omega_2 \setminus \p \Omega_1} = 0$. It follows that $\ell(u) = 0$ as required.
\end{proof}

\begin{proof}[Proof of Theorem \ref{thm_linearized_lone} for $f \in L^r(\Omega)$, $r > 1$]
The proof of \cite[Theorem 1.1]{ferreira2009linearized} for $f \in L^{\infty}(\Omega)$ proceeds in three steps:
\begin{enumerate}
\item[1.]
Reduction to a case where $\Omega$ is strictly convex near some $x_0 \in \Gamma$.
\item[2.]
Local result showing that $f = 0$ near $x_0$.
\item[3.]
Iteration of the local result to show that $f=0$ everywhere.
\end{enumerate}
Step 1 works equally well for $f \in L^1(\Omega)$. We may thus assume that we are in the setting in \cite[Section 3]{ferreira2009linearized} where $x_0 = 0$, $T_0 (\p \Omega) = \{ x_1 = 0 \}$, 
\[
\Omega \subset \{ x \in \mR^n \,:\, \abs{x+e_1} < 1 \}, \quad \Gamma \subset \{ x \in \p \Omega \,:\, x_1 \geq -2c \},
\]
for some $c > 0$. (Note that our $\Gamma$ corresponds to $\p \Omega \setminus \Gamma$ in \cite{ferreira2009linearized}.)

Let us indicate the necessary changes in \cite[Section 3]{ferreira2009linearized} in order to do Step 2 for $f \in L^1(\Omega)$. We consider harmonic functions 
\[
u(x,\zeta) = e^{-\frac{ix \cdot \zeta}{h}} + w(x,\zeta)
\]
where $h > 0$ is small, $\zeta \in \mC^n$ satisfies $\zeta \cdot \zeta = 0$, and $w$ solves 
\[
\Delta w = 0 \text{ in $\Omega$}, \qquad w|_{\p \Omega} = -e^{-\frac{ix \cdot \zeta}{h}} \chi|_{\p \Omega}
\]
where $\chi \in C^{\infty}(\p \Omega)$ satisfies $\chi=1$ for $x_1 \leq -2c$ and $\supp(\chi) \subset \{ x_1 \leq -c \}$. Then $\supp(u) \subset \Gamma$. Now, instead of using a $H^1$ estimate for $w$ as in \cite[formula (3.6)]{ferreira2009linearized}, we use an $L^{\infty}$ estimate (i.e.\ maximum principle):
\[
\norm{w}_{L^{\infty}(\Omega)} \leq \lVert e^{-\frac{ix \cdot \zeta}{h}} \chi \rVert_{L^{\infty}(\p \Omega)} \leq e^{-\frac{c}{h} \im\,\zeta_1} e^{\frac{1}{h} \abs{\im \,\zeta'}} \quad \text{when $\im\,\zeta_1 \geq 0$.}
\]
Here we write $\zeta = (\zeta_1, \zeta')$ where $\zeta' \in \mC^{n-1}$. Since we have 
\[
\int_{\Omega} f(x) u(x,\zeta) u(x,\eta) \,dx, \qquad \zeta \cdot \zeta = \eta \cdot \eta = 0,
\]
we get for $\im\,\zeta_1, \im\,\eta_1 \geq 0$ the estimate 
\begin{multline*}
\left| \int_{\Omega} f(x) e^{-\frac{ix \cdot (\zeta+\eta)}{h}} \,dx \right| \leq \norm{f}_{L^1} ( \lVert e^{-\frac{ix \cdot \zeta}{h}} \rVert_{L^{\infty}} \norm{w(x,\eta)}_{L^{\infty}} \\ %\norm{f}_{L^1} (  e^{-\frac{ix \cdot \zeta}{h}}_{L^{\infty} \norm{w(x,\eta)}_{L^{\infty}} \\
 +  \lVert e^{-\frac{ix \cdot \eta}{h}} \rVert_{L^{\infty}} \norm{w(x,\zeta)}_{L^{\infty}} + \norm{w(x,\zeta)}_{L^{\infty}} \norm{w(x,\eta)}_{L^{\infty}}) \\
  \leq 3 \norm{f}_{L^1} e^{-\frac{c}{h} \min(\im \,\zeta_1, \im \,\eta_1)} e^{\frac{1}{h} (\abs{\im \,\zeta'} + \abs{\im \,\eta'})}.
\end{multline*}
Then, using the same notations as in \cite{ferreira2009linearized}, formula (3.8) in \cite{ferreira2009linearized} gets replaced by 
\begin{equation} \label{lin_main_estimate}
\left| \int_{\Omega} f(x) e^{-\frac{ix \cdot (\zeta+\eta)}{h}} \,dx \right| \leq C \norm{f}_{L^1(\Omega)} e^{-\frac{ca}{2h}} e^{\frac{2C\eps a}{h}}.
\end{equation}

We now proceed to Section 4 in \cite{ferreira2009linearized}. Note that for $f \in L^1(\Omega)$, equation (4.1) in \cite{ferreira2009linearized} is replaced by 
\[
\abs{Tf(z)} \leq e^{\frac{1}{2h} \abs{\im \,z}^2} \norm{f}_{L^1(\Omega)}
\]
for $z \in \mC^n$. Using the condition $\supp(f) \subset \{ x_1 \leq 0 \}$, equation (4.2) in \cite{ferreira2009linearized} is replaced by 
\[
\abs{Tf(z)} \leq e^{\frac{1}{2h} (\abs{\im \,z}^2 - (\re \,z_1)^2)} \norm{f}_{L^1(\Omega)}
\]
for $\re\,z_1 \geq 0$. Finally, using \eqref{lin_main_estimate}, equation (4.7) in \cite{ferreira2009linearized} is replaced by 
\[
\abs{Tf(z)} \leq C \norm{f}_{L^1(\Omega)} e^{\frac{1}{2h} (\abs{\im \,z}^2 - \abs{\re \,z}^2 - \frac{ca}{2})}.
\]
From the three estimates above we see that the estimate (4.8) in \cite{ferreira2009linearized} holds with $h^{-1} \norm{f}_{L^{\infty}(\Omega)}$ replaced by $\norm{f}_{L^1(\Omega)}$. The proof of Step 2 is now completed as in \cite[Section 4]{ferreira2009linearized}.

It remains to explain how to do Step 3 for $f \in L^r(\Omega)$, $r > 1$. Inspecting the arguments in \cite[Section 2]{ferreira2009linearized}, it is sufficient to prove that the set described in \cite[formula (2.2)]{ferreira2009linearized} is dense for the $L^p(\Omega_1)$ topology, for any $p < \infty$, in the subspace of harmonic functions $u \in C^{\infty}(\ol{\Omega}_1)$ such that $u|_{\p \Omega_1 \cap \p \Omega_2} = 0$. This follows from Lemma \ref{lemma_runge}.
\end{proof}

In the remainder of this section we prove Theorem \ref{thm_linearized_lone} for $f \in L^1(\Omega)$. We have seen in the proof above that Steps 1 and 2 already work for $f \in L^1(\Omega)$, so it is enough to consider Step 3 and an analogue of the Runge approximation argument of Lemma \ref{lemma_runge} but in the $L^{\infty}$ norm. Such a result was proved in \cite[Lemma 2.6]{kian2020partial}, but here we give an alternative argument where $\Omega_2$ will have nonsmooth boundary.

Let us briefly explain the rationale behind this. In the $L^p$ approximation proof above we used \eqref{runge_ibp}, where $w$ solves $\Delta w = E_0 v$ in $\Omega_2$ with $w|_{\p \Omega_2} = 0$, and $Eu$ is a sufficiently regular extension of $u \in S$. For approximation in $L^{\infty}$, the quantity $v$ will be in the dual of $L^{\infty}$ (i.e.\ a finitely additive measure) and one would require additional work to make sense of the normal derivative $\p_{\nu} w|_{\p \Omega_2}$ in \eqref{runge_ibp}. We will instead construct the extension $Eu$ so that $Eu|_{\p \Omega_2} = 0$. When $\Omega_2$ is a smooth domain such an extension does not exist in general, but for suitable nonsmooth domains it does.

Let $0 < \alpha < 1$. We say that a domain $\Omega$ has a \emph{$C^{1,\alpha}$ edge singularity} along a subset $E \subset \p \Omega$ if for any $x_0 \in E$ there is a neighborhood $U$ of $x_0$ in $\mR^n$ and a diffeomorphism $F: U \to \tilde{U} \subset \mR^n$ such that $F(x_0) = 0$ and one has bijective maps 
\begin{align*}
F&: \Omega \cap U \to \{ (x_1, x_2, x') \,:\, x_2 < \psi(x_1) \} \cap \tilde{U}, \\
F&: E \cap U \to \{ (0,0,x') \} \cap \tilde{U},
\end{align*}
where $\psi: \mR \to \mR$ is smooth away from $0$ with $\psi(t) = 0$ for $t \leq 0$, and the function $\psi(t)/t^{1+\alpha}$ is smooth in $[0,\infty)$ and nonvanishing at $0$. Here we write $x = (x_1, x_2, x')$ for points $x \in \mR^n$, where $x' \in \mR^{n-2}$.

\begin{lem} \label{runge apprx}
Let $\Omega_1\subset \Omega_2 \subset \mR^n$ be bounded open sets, let $\Omega_1$ have smooth boundary, and assume that $\Omega_2$ has smooth boundary except at $\p (\p \Omega_1 \cap \p \Omega_2)$ where it has a $C^{1,\alpha}$ edge singularity. Also assume that if $x_0$ is a point on the edge and $\Omega_2$ is locally near $x_0$ given by $\{ x_2 < \psi(x_1) \}$ as above, then $\Omega_1$ is locally near $x_0$ given by $\{ x_2 < \eta(x_1) \}$ where $\eta \in C^{\infty}(\mR)$ satisfies $\eta(t) = 0$ for $t \leq 0$ and $\eta(t) < \psi(t)$ for $t > 0$.

%{\color{red} Could we simply say $\{x_2<0\}$ and $\{x_2 < x_1^{1+\alpha}\}$? }\tbl{[MS: the way I thought of this is that $\Omega$ is locally $\{ x_2 = 0 \}$, so $\Omega_1$ has to be below $\{ x_2 = 0 \}$ in the part of interest.]}

Let $G_{\Omega_2}(x,y)$ be the Dirichlet Green's kernel associated with $\Omega_2$. Then for $1 < p < 1+2/\alpha$ the set 
\[
R = \left\{ \int_{\Omega_2} G_{\Omega_2}(\cdot, y) a(y) \,dy  \,:\, \, a\in C^\infty_c(\Omega_2), \ \mathrm{supp}(a)\subset \Omega_2 \setminus \ol{\Omega}_1 \right\}
\]
is a dense subspace, with respect to the $W^{1,p}(\Omega_1)$ topology, in the space $S$ of harmonic functions $u \in C^{\infty}(\ol{\Omega}_1)$ with $u|_{\partial\Omega_1 \cap \partial \Omega_2} = 0$.
\end{lem}
%{\color{red} Presumably you have to allow more general geometry in the $x'$ direction. So instead of saying ``given by $\{ x_2 < \psi(x_1) \}$ as above'', maybe we say ``diffeomorphic to''.  }

We begin with an extension result in a model case.

\begin{lem}
Let $u \in C^1_c(\{x_2 \leq 0 \})$ be such that $u|_{x_2=0}$ is supported in $\{ x_1 \geq 0 \}$. Given $0 < \alpha < 1$ and $\delta > 0$, there is an extension $\tilde{u}$ of $u$ to $\mR^n$ such that $\tilde{u} \in W^{1,p}(\mR^n)$ for $1 \leq p < 1+2/\alpha$ and 
\[
\mathrm{supp}(\tilde{u}) \subset \{x_2\leq 0\}\cup \{ (x_1,x_2,x') \,:\, x_1 > 0, \,x_2 \leq \delta x_1^{1+\alpha}\}.
\]
\end{lem}
\begin{proof}
Let $Eu$ be any $C^1_c(\mR^n)$ extension of $u$, and define 
\[
\tilde{u}(x) := \begin{cases} u(x), & x_2 \leq 0, \\ \chi(x_2/x_1^{1+\alpha}) Eu(x), & x_1 > 0, \,x_2 > 0, \\ 0 & \text{elsewhere}, \end{cases}
\]
where $\chi = \chi_{\delta} \in C^{\infty}_c(\mR)$ satisfies $\chi(t) = 1$ for $\abs{t} \leq \delta/2$ and $\chi(t) = 0$ for $\abs{t} \geq \delta$. Then $\tilde{u}$ is $C^1$ away from $\{ x_1 = x_2 = 0 \}$ and continuous in $\mR^n$ since $u(0,0,x') = 0$. If $\varphi \in C^{\infty}_c(\mR^n)$, we may compute the weak derivatives of $\tilde{u}$ via 
\begin{align*}
\int_{\mR^n} \tilde{u} \p_j \varphi \,dx &= \lim_{\eps \to 0} \int_{\abs{(x_1,x_2)} > \eps} \tilde{u} \p_j \varphi \,dx \\
 &= \lim_{\eps \to 0} \int_{\abs{(x_1,x_2)} = \eps} \tilde{u} \varphi \nu_j \,dS - \lim_{\eps \to 0} \int_{\abs{(x_1,x_2)} > \eps} \p_j \tilde{u} \varphi \,dx.
\end{align*}
The first term on the right vanishes by continuity of $\tilde{u}$, and hence the weak derivative $\p_j \tilde{u}$ is given by 
\[
\p_j \tilde{u}(x) := \begin{cases} \p_j u(x), & x_2 \leq 0, \\ \p_j(\chi(x_2/x_1^{1+\alpha}) Eu(x)), & x_1 > 0, \,x_2 > 0, \\ 0 & \text{elsewhere}. \end{cases}
\]
It is enough to verify that $\p_j(\chi(x_2/x_1^{1+\alpha}) Eu(x)) \in L^p(\{ x_1, x_2 > 0 \})$ for $p < 1+2/\alpha$. This is clear for $j \geq 3$. For $j=2$ we compute 
\[
\p_2(\chi(x_2/x_1^{1+\alpha}) Eu) = \chi(x_2/x_1^{1+\alpha}) \p_2 Eu + \chi'(x_2/x_1^{1+\alpha}) x_1^{-1-\alpha} Eu
\]
The first term is in $L^p$. For the second term we use that $\abs{t} \sim \delta$ on $\supp(\chi')$, which gives that $\abs{x_2} \sim \delta x_1^{1+\alpha}$ on $\supp(\chi'(x_2/x_1^{1+\alpha}))$. Using the fact that $Eu \in C^1_c(\mR^n)$, in $\{ x_1, x_2 > 0 \}$ we have 
\begin{multline*}
\left\lvert \frac{\chi'(x_2/x_1^{1+\alpha})}{x_1^{1+\alpha}} Eu(x) \right\rvert = \left\lvert \frac{\chi'(x_2/x_1^{1+\alpha})}{x_1^{1+\alpha}} (Eu(x_1,x_2,x') - Eu(0,0,x')) \right\rvert \\
 \leq C \left\lvert \frac{\chi'(x_2/x_1^{1+\alpha})}{x_1^{1+\alpha}}(\abs{x_1} + \abs{x_2}) \right\rvert \leq C \left\lvert \frac{\chi'(x_2/x_1^{1+\alpha})}{x_1^{\alpha}} \right\rvert.
\end{multline*}
The last quantity is in $L^p(\{ x_1, x_2 > 0 \} \cap B_1)$ when $p < 1+2/\alpha$ since $x_2 \sim \delta x^{1+\alpha}$ in the integration set. The behaviour of $\p_1 \tilde{u}$ is even better. This proves that $\tilde{u} \in W^{1,p}$ for $p < 1+2/\alpha$.
\end{proof}

\begin{cor}
\label{extension out of general domain}
Let $\Omega_1\subset \Omega_2$ be bounded open sets having the properties stated in Lemma \ref{runge apprx}. Suppose that $u\in C^1(\ol{\Omega}_1)$ has vanishing trace on $\partial\Omega_1\cap\partial\Omega_2$. Then for $1 \leq p < 1+2/\alpha$, $u$ has an extension $\tilde u\in W^{1,p}(\R^n)$ supported in $\overline{\Omega}_2$.
\end{cor}
\begin{proof}
Since $u|_{\partial\Omega_1\cap\partial\Omega_2} = 0$, by using smooth cutoff functions it suffices to construct the extension near any point of the submanifold $\partial(\partial\Omega_1\cap \partial\Omega_2)$.  By the assumptions on $\Omega_1$ and $\Omega_2$, we can use a diffeomorphism to map a neighbourhood of any $x_0 \in \partial(\partial\Omega_1\cap \partial\Omega_2)$ into $\R^n = \{(x_1,x_2,x')\}$ where locally $\Omega_1 = \{x_2< \eta(x_1) \}$ and  $\Omega_2 = \{ x_2 < \psi(x_1) \}$ with $\psi$ and $\eta$ having the properties stated above. Choose new coordinates so that $y_1 = x_1$, $y_2 = x_2 - \eta(x_1)$, and $y' = x'$. Then locally $\Omega_1 = \{  y_2 < 0 \}$ and $\Omega_2 = \{ y_2 < \psi_1(y_1) \}$ where $\psi_1(t) = \psi(t) - \eta(t)$ is such that $\psi_1(t)/t^{1+\alpha}$ is smooth in $[0,\infty)$ and positive at $0$, using that $\eta$ vanishes to infinite order at $0$. Now apply the previous lemma in the $y$ coordinates.
\end{proof}

\begin{proof}[Proof of Lemma \ref{runge apprx}]
We only indicate the modifications required in the proof of Lemma \ref{lemma_runge}. Now $\ell$ is a bounded linear functional on $W^{1,p}(\Omega_1)$, and hence it is represented by $v \in W^{-1,p'}(\mR^n)$ with $\supp(v) \subset \ol{\Omega}_1$. We now wish to solve the Dirichlet problem 
\[
\Delta w = v|_{\Omega_2} \text{ in $\Omega_2$}, \qquad w|_{\p \Omega_2} = 0.
\]
Since $\Omega_2$ is a $C^{1,\alpha}$ domain and $v|_{\Omega_2} \in W^{-1,p'}(\Omega_2)$, by \cite[Theorem 1.1]{JerisonKenig1995} there is a solution $w \in W^{1,p'}_0(\Omega_2)$ whenever $1 < p < \infty$. As in \eqref{w_vanishing} we obtain $w|_{\Omega_2 \setminus \ol{\Omega}_1} = 0$.

Now for any $u \in S$ one has 
\[
\ell(u) = \langle v, \tilde{u} \rangle_{\mR^n}
\]
where $\tilde{u}$ is any function in $W^{1,p}(\mR^n)$ with $\tilde{u}|_{\Omega_1} = u$. If $p < 1+2/\alpha$ and we choose $\tilde{u}$ to be the extension given in Corollary \ref{extension out of general domain}, we have $\supp(\tilde{u}) \subset \ol{\Omega}_2$ and hence we may consider $\tilde{u}$ as an element of $W^{1,p}_0(\Omega_2)$. On the other hand, the facts that $w \in W^{1,p'}_0(\Omega_2)$ and $w|_{\Omega_2 \setminus \ol{\Omega}_1} = 0$ imply that there is a sequence $w_j \in C^{\infty}_c(\Omega_1)$ with $w_j \to w$ in $W^{1,p'}(\Omega_2)$. It follows that
\[
\ell(u) = \langle v|_{\Omega_2} , \tilde{u} \rangle_{\Omega_2} =  \langle \Delta w, \tilde{u} \rangle_{\Omega_2} = \lim \,\langle \Delta w_j, \tilde{u} \rangle_{\Omega_2} = \lim \,\langle w_j, \Delta \tilde{u} \rangle_{\Omega_2}.
\]
Since $u \in S$ we have $\Delta \tilde{u} = 0$ in $\Omega_1$, and thus the last expression vanishes using that $w_j \in C^{\infty}_c(\Omega_1)$. We have shown that $\ell|_S = 0$, which concludes the proof.
\end{proof}

It now remains to complete Step 3 in the proof of Theorem \ref{thm_linearized_lone} for $f \in L^1(\Omega)$. We follow the argument of \cite[Section 2]{ferreira2009linearized} with minor modifications. Let $x_1\in \Omega$ and let $\theta: [0,1] \to \overline \Omega$ be a smooth curve so that $\theta(0)$ is the only point of $\theta([0,1])$ on $\partial \Omega$ and $\theta'(0)$ is normal to $\p \Omega$. Define 
\[
\Theta_\eps(t) :=\{x\in \overline\Omega \,:\, d(x,\theta([0,t]))<\eps \}.
\]
Thanks to Step 2, there exists $\eps>0$ such that $f= 0$ in $\Theta_\eps(0) \cap \Omega$. We may further decrease $\eps$ so that $U :=\Theta_\eps(1) \cap \partial \Omega$ is a connected small open neighbourhood on $\partial \Omega$ containing $x_0$. Locally near $x_0$ we may work in coordinates so that $x_0=0$, $\Omega = \{ x_n < 0 \}$ near $x_0$, and $\ol{B_{\delta}(0)} \cap \{ x_n = 0 \} \subset U$. We also ask that $\delta>0$ is small enough such that $\partial\Omega\setminus \Gamma \subset \subset \partial \Omega \setminus B_{\delta}(0)$. Choose $\alpha < \frac{2}{n-1}$, so that $1+2/\alpha > n$, and in the previous coordinates choose 
\[
\Omega_2 := \Omega\cup C_{\alpha}
\]
where $C_{\alpha}$ is a set in $\{ x_n \geq 0 \}$ so that $\Omega_2$ will have a $C^{1,\alpha}$ edge singularity along $\p B_{\delta}(0) \cap \{ x_n = 0 \}$.

For $\eps>0$ fixed above, define 
\[
I := \{t\in [0,1] \,:\, f = 0 \text{ a.e.\ in } \Theta_\eps(t) \cap \Omega\}.
\]
This is clearly a nonempty set which is closed. We now need to show that it is open.

To this end, suppose $t_0\in I \cap (0,1)$. By our choice of $\theta(\cdot)$ and $\eps$, there is an open set $\Omega_1\subset \Omega$ with smooth boundary so that $\partial\Omega\setminus B_{\delta}(0) = \partial\Omega \cap \partial\Omega_1$,  $\partial (\partial\Omega\cap\partial \Omega_1) = \p B_\delta(0) \cap \{x_n = 0 \}$, and so that near any point of the edge $\p (\p \Omega_1 \cap \p \Omega_2) = \p B_{\delta}(0) \cap \{ x_n = 0 \}$ the sets $\Omega_1$ and $\Omega_2$ satisfy the conditions in Lemma \ref{runge apprx}. We also ask that $\Omega \setminus \Theta_\eps(t_0) \subset \Omega_1 \subset \Omega \setminus \theta([0,t_0])$, that $\Omega \setminus \ol{\Omega}_1$ is connected, and that $\partial B_\eps(\theta(t_0))\cap \partial \Theta_\eps(t_0) \subset \partial\Omega_1$. 

Let $G(x,y)$ be the Dirichlet Green's function associated to the domain $\Omega_2$. Consider the expression 
\[
\int_{\Omega_1} f(y) G(x,y) G(t,y) \,dy
\]
as a function of both $x,t \in \Omega_2\setminus \overline \Omega_1$. Since $f=0$ in $\Omega\setminus \Omega_1$, we have that 
\[
\int_{\Omega_1} f(y) G(x,y) G(t,y)\,dy = \int_{\Omega} f(y) G(x,y) G(t,y)\,dy.
\]
For $x,t\in \Omega_2\setminus \ol{\Omega}$, $y\mapsto G(x,y)$ and $y\mapsto G(t,y)$ are harmonic functions in $\Omega$ which vanish on $\partial\Omega_1 \cap \partial\Omega_2 \supset \partial\Omega\setminus \Gamma$. So by our assumption that $f$ is orthogonal to products of such harmonic functions, we have 
\[
0 = \int_{\Omega_1} f(y) G(x,y) G(t,y)\,dy
\]
for $x,t\in \Omega_2\setminus \ol{\Omega}$. By unique continuation,
\[
0 = \int_{\Omega_1} f(y) G(x,y) G(t,y)\,dy
\]
for $x,t\in \Omega_2\setminus \ol{\Omega}_1$. By integrating in $x$ and $t$ against smooth functions supported in $\Omega_2 \setminus \ol{\Omega}_1$ we have that
\[
0 = \int_{\Omega_1}fuv
\]
for all $u$ and $v$ harmonic in $\Omega_1$ of the form $\int_{\Omega_2}G(\cdot, y) a(y) \,dy$ with $\supp(a) \subset \Omega_2 \setminus \ol{\Omega}_1$. By the density result of Lemma \ref{runge apprx} (since $1+2/\alpha > n$, we have density in $W^{1,p}$ for some $p > n$ and hence in $L^{\infty}$), this means that 
\[
0= \int_{\Omega_1} fuv
\]
for all $u$ and $v$ harmonic in $\Omega_1$ and vanishing on $\partial\Omega_1\cap\partial\Omega_2$. By applying the local result in Step 2, we can conclude that $f$ vanishes in an open subset containing $\p B_\eps(\theta(t_0)) \cap \p \Theta_{\eps}(t_0)$. This shows that $t_0$ is an interior point of $I$, showing that $I$ is open and concluding the proof.

\bibliographystyle{alpha}
\bibliography{ref}

\end{document}